\documentclass[12pt,reqno]{amsart}

\usepackage[pagewise]{lineno}

\usepackage{paralist}
\usepackage[english]{babel}
\usepackage{amsmath,amsthm}

\usepackage{amssymb}
\usepackage{graphicx}

\usepackage[
  pagebackref=true,
  colorlinks,
]{hyperref}

  \usepackage{supertabular}
  \usepackage[all,cmtip]{xy}
  \usepackage{graphics} 
  \usepackage{epsfig} 
\hypersetup{urlcolor=blue, citecolor=red}

\DeclareMathOperator{\Tr}{Tr}
\DeclareMathOperator{\dive}{div}

\DeclareMathOperator{\psl}{PSL}
\DeclareMathOperator{\gl}{GL}
\DeclareMathOperator{\aut}{Aut}

\DeclareMathOperator{\mon}{Mon}
\DeclareMathOperator{\SL}{SL}
\DeclareMathOperator{\End}{End}
\DeclareMathOperator{\sym}{Sym}

\newtheorem{defi}{Definition}[section]
\newtheorem{lem}[defi]{Lemma}
\newtheorem{prop}[defi]{Proposition}

\newtheorem{teo}[defi]{Theorem}
\newtheorem{cor}[defi]{Corollary}
\newtheorem{obs}[defi]{Remark}

\newtheorem{THM}{Theorem}

\newcounter{qqq}
\vspace{.5cm} 

\newcommand{\Proj}{\mathbb{P}}

\newcommand{\con}{\nabla}

\newcommand{\Po}{\mathcal{P}}
\newcommand{\W}{\mathcal{W}}

\newcommand{\R}{\mathbb{R}}
\newcommand{\C}{\mathbb{C}}
\newcommand{\Z}{\mathbb{Z}}
\newcommand{\N}{\mathbb{N}}
\newcommand{\Hip}{\mathbb{H}}

\newcommand{\Fo}{\mathcal{F}}
\newcommand{\Go}{\mathcal{G}}
\newcommand{\Do}{\mathcal{D}}
\newcommand{\Ho}{\mathcal{H}}

\newcommand{\Ol}{\mathcal{O}}

\newcommand{\DR}{\mathcal{DR}}

\newcommand{\im}{\mathrm{im}\,}

\begin{document}

 \title[Affine connections and Riccati distributions]{Affine connections on complex 
 compact surfaces and Riccati distributions}
 
 \author{Ruben Lizarbe}
 
 \address{UERJ, Universidade do Estado do Rio de Janeiro, Rua S\~ao Francisco Xavier, 
 524, Maracan\~a, 20550-900, Rio de Janeiro, Brazil. }
 
\email{ruben.monje@ime.uerj.br}

\makeatletter
\@namedef{subjclassname@2020}{\textup{2020} Mathematics Subject Classification}
\makeatother
\date{\today}


\maketitle

\begin{abstract}
Let $M$ be a complex surface. We show that there is a one-to-one 
correspondence between torsion-free affine connections on $M$ and Riccati distributions 
on $\Proj(TM)$. Furthermore, if $M$ is compact, then this correspondence
induces a one-to-one corres\-pondence between affine structures on $M$ and 
Riccati foliations on $\Proj(TM)$. As applications of this result, we classify  
the regular $k$-webs on compact complex surfaces for $k\geq 3$, and we also get a new proof of the classification of regular pencils of foliations on compact complex surfaces.
\end{abstract}

\smallskip
\noindent


\section{Introduction}

There are numerous papers on compact complex surfaces that admit 
holomorphic affine connections (see 
\cite{Gunning}, \cite{Vitter}, \cite{Suwa}, \cite{Maehara}, \cite{Kob1}, 
\cite{Klingler} and \cite{Dumitrescu}).  
In \cite{Kob1}, Inoue, Koba\-yashi  
and Ochiai give a complete list of all compact complex (connected) surfaces
admitting affine holomorphic connections which are not necessarily flat. These surfaces 
are shown to be biholomorphic (up to a finite covering) to complex tori, 
primary Kodaira surfaces, affine Hopf surfaces, 
Inoue surfaces or elliptic surfaces over Riemann surfaces of genus 
$g\geq 2$ with odd first Betti number. Moreover, it is proved in \cite{Kob1} that 
all these surfaces have torsion-free flat affine 
connections that are equivalent to holomorphic affine structures, that is, to atlas with values 
in open subsets of $\C^2$ whose change of coordinate maps are locally constant mappings in the affine group 
$\gl (2, \C) \ltimes \C^2$. These surfaces are said to be affine. It is also known that 
they are quotients of a domain in $\C^2$ by a group consisting of affine 
transformations.

In \cite{Klingler} Klingler classifies holomorphic affine structures and holomorphic projective structures on compact complex surfaces. Also, in \cite{Dumitrescu} Dumitrescu classifies torsion-free holomorphic affine connections and shows that any normal holomorphic projective connection on a compact complex surface
has zero curvature.  Finally, in his thesis \cite{Zhao} Zhao has studied affine structures and birational structures. 

The aim of this paper is to study the correspondence between affine connections 
on complex surfaces and Riccati distributions. We list our main results.

\begin{THM} \label{T1}
Let $M$ be a complex surface. 
There is a one-to-one co\-rrespondence between
torsion-free affine connections on $M$ and Riccati distributions on $\Proj(TM)$. Furthermore, we 
have a one-to-one correspondence between 
affine structures on $M$ and parallelizable Riccati foliations on $\Proj(TM)$.   
\end{THM}

By a Riccati distribution on $\Proj(TM)$ we mean a codimension one re\-gu\-lar distribution on 
$\Proj(TM)$ that is transversal to every fiber of $\Proj^1$-bundle $\pi:\Proj(TM)\to M$.  

\begin{THM} \label{T2}
If $M$ is a compact complex surface, then a Riccati distribution on $\Proj(TM)$ is parallelizable. 
In particular,  
we have a one-to-one corres\-pondence between 
affine structures on $M$ and Riccati foliations on $\Proj(TM)$.  
\end{THM}

Since important examples of Riccati foliations are induced by pencils and $k$-webs, we can use 
the previous theorems to study regular pencils and $k$-webs on complex surfaces. 
A regular pencil of foliations $\Po=\{\Fo_t \}_{t\in\Proj^1}$ on $M$ is locally defined by 
a pencil of 1-forms $\Fo_t =\{\omega_t=0\}_{t\in\Proj^1}$ 
with $\omega_t=\omega_0 + t\omega_{\infty}$ and $\omega_0 \wedge  \omega_{ \infty}\neq 0$, thait 
is, $\Fo_0$ and $\Fo_{\infty}$ are transversal (see \cite{Alcides} or \cite{Puchuri}). 
We are able to give a new proof of the classification of regular pencils on compact complex 
surfaces proved by Puchuri \cite{Puchuri}.  

\begin{THM}\label{T3} 
If $M$ is a compact complex surface and $\Po$ is a re\-gu\-lar pencil on $M$, then $M$ is either 
\begin{enumerate}
\item a complex torus $M=\C^2/\Gamma$, where $\Gamma$ is a lattice in $\C^2$,  or 
\item a Hopf surface $M$,  
with universal covering space $\C^2\setminus \{(0,0)\}$. Its fundamental
group is isomorphic to $\Z\oplus \Z/l\Z$ , for some integer $l\geq 1$; it is generated by
a diagonal automorphism $(x, y) \mapsto (ax, ay)$ with $|a| < 1$, and a diagonal
automorphism $(x, y) \mapsto (\varepsilon x,\varepsilon y)$ where $\varepsilon$ is a primitive 
$l$-th root of $1$.  
\end{enumerate}

Moreover, $\Po$ is the only regular pencil on $M$, 
and it is generated by $\{dx+tdy\}_{t\in\Proj^1}$ on $\C^2$ (in the first case) or on 
$\C^2\setminus \{(0,0)\}$ (in the second case). 
\end{THM}

A regular $k$-web $\W = \Fo_1\boxtimes \cdots\boxtimes \Fo_l$ on $M$ 
is given locally by $k$ regular foliations 
$\Fo_i$ in general position, that is, pairwise transversal (see \cite{Pereira}  or \cite{Robert}). When $M$ is compact and $k\geq 3$, 
we can classify all the regular k-webs on it. 

\begin{THM}\label{T4} 
If $M$ is a compact complex surface and $\W$ is a regular $k$-web on $M$, $k\geq 3$,   
then $M$ is either 
\begin{enumerate} 
\item a complex torus $\C^2/\Gamma$, where $\Gamma$ is a lattice in $\C^2$, or 
\item a hyperelliptic surface $(E\times F)/G$, where $E$ and $F$ are elliptic curves, and 
$G$ is a finite abelian subgroup of $E$ (acting on $E$ by translations), or
\item a Hopf surface $M$, with universal covering space $\C^2\setminus \{(0,0)\}$. Its fundamental
group is isomorphic to $\Z\oplus \Z/l\Z$ , for some integer $l\geq 1$; it is generated by
a diagonal automorphism $(x, y) \mapsto (ax, \xi ay)$ with $|a| < 1$ and $\xi$ is primitive 
$j$-th root of $1$, for some integer $j\geq 1$, and a diagonal
automorphism $(x, y) \mapsto (\varepsilon_1 x,\varepsilon_2 y)$ where $\varepsilon_1, \varepsilon_2$ 
are primitive $l$-th roots of $1$. 

\end{enumerate}
Moreover, up to a finite cover, the $k$-web $\W$ is contained $\Po$, where $\Po$ 
denotes the only pencil on $M$ from Theorem \ref{T3}.    
\end{THM}

In Section \ref{section2} we introduce the notion of a Riccati connection on a complex manifold of dimension 
$n\geq2$. We also define the trace and the curvature of a Riccati connection. The main result in this section 
guarantees that the existence of a Riccati connection on a compact K\"ahler manifold yields a relationship between its Chern classes (see Theorem \ref{Riccaticlasses} and Proposition \ref{restricsurfaces}). 
We shall see that in the case of dimension 2, the K\"ahler hypothesis is not necessary. 

In Section \ref{section3} given a complex surface $M$ we establish a natural equiva\-lence between 
reduced Riccati connections on $M$ and Riccati distributions on $\Proj(TM)$. We also introduce 
the notion of the curvature of a Riccati distribution and show that if $M$ 
is compact, then a Riccati distribution is parallelizable, that is, it has zero curvature. 
Theorem \ref{T1} and Theorem \ref{T2} are proved in this section. 
By using the geometric des\-cription of all the affine compact complex surfaces  
(see the computations in \cite[Th\'eor\`eme 1.2]{Klingler} with a fixed compact complex 
surface $M$, where the set of complex affine structures on $M$ compatible 
with its analytical structure are determined), we were able to calculate the monodromy of the 
Riccati foliations on $\Proj(TM)$.

In Section \ref{section4} we apply the results of Section \ref{section3} 
to prove Theorem \ref{T3} and Theorem \ref{T4}. 

\section{Riccati connections}\label{section2}
The following definition was motivated by R. Molzon and K. Pinney Mortensen  
\cite{Molzon}.  
Let $M$ be a complex manifold of dimension $n\geq2$. A {\it Riccati connection} on M is a 
$\C$-bilinear map \[ \DR:TM\times TM \rightarrow TM,\]
satisfying 
\begin{enumerate} 
\item $\DR_{fX}Y=f\DR_{X}Y-\frac{1}{n}X(f)Y$, 
\item $\DR_{X}fY=f\DR_{X}Y+X(f)Y$,  
\end{enumerate} 
for any local holomorphic functions $f$ and any local vector fields $X,Y$ . 

In coordinates $(x_1,\ldots,x_n)\in U\subset \C^n$, a trivialization of $TU$ is given by the basis 
$(\partial_{x_1},\ldots,\partial_{x_n})$ and the Riccati connection is given by
\[\DR_X(Y)=d(Y)+\theta Y-\frac{1}{n} \dive (X)Y, 
\]
where $\theta$ is the {\it matrix of 1-forms associated with the Riccati connection $\DR$} and $\dive$ represents the divergence operator.  

Letting $(\varphi_{\alpha} ,U_{\alpha})$, $(\varphi_{\beta},U_{\beta})$ be two local systems of 
coordinates on $M$ with $U_{\alpha}\cap U_{\beta}\neq \emptyset$, and  denoting the corresponding 
change of coordinates by $\varphi_{\alpha\beta}=\varphi_{\alpha}\circ \varphi_{\beta}^{-1}$, we have 
\begin{equation}\label{rc} 
\theta_{\alpha}=g_{\alpha\beta}\theta_{\beta}g_{\alpha\beta}^{-1}
-dg_{\alpha\beta}g_{\alpha\beta}^{-1}+\frac{1}{n}\Tr(dg_{\alpha\beta}g_{\alpha\beta}^{-1})I
\end{equation}
where $g_{\alpha\beta}$ represents the Jacobian matrix of $\varphi_{\alpha\beta}$ and $\Tr$ the trace 
ope\-rator. 

Note that the trace of $\theta$ represents a 1-form on $M$ due to equation (\ref{rc}).

Now we will describe two invariants defined from a Riccati connection.   

\subsection{Trace  and curvature of a Riccati connection}

\begin{defi}
{\rm   The {\it trace} of a Riccati connection $\DR$, denoted by 
$\Tr( \DR)$, is the 1-form defined as the trace of a matrix of 1-forms on $M$
associated with $\DR$.
We say that $ \DR$ is a {\it reduced 
Riccati connection} if its trace is zero, i.e., $ \Tr( \DR)=0$.    
}
\end{defi}

\begin{obs}
{\rm Every Riccati connection $\DR$ on $M$ determines a 
reduced Riccati connection $\tilde{\DR}$  on $M$ as follows
\[
\tilde {\DR}_X Y=\DR_X Y-\frac{\Tr(\DR)(X)}{n}Y. 
\]
Two Riccati connections $\DR$ and $\hat\DR$ on $M$ determine the same 
reduced Riccati connection if and only if 
there is a 1-form $\gamma$ on $M$ such that $\hat\DR_X Y-\DR_X Y=\gamma(X)Y$.
}
\end{obs}

We introduce the notion of curvature of a Riccati connection follo\-wing ideas similar 
to those of Kato \cite{Kato}.  

Using the same notation as in (\ref{rc}), we define  
\begin{equation}\label{equationcurvatureRiccati}
W_{\alpha}:=d\theta_{\alpha}+\theta_{\alpha}\wedge \theta_{\alpha}.
\end{equation}

The matrices of 2-forms $W_{\alpha}$ is called the {\it curvature} of the Riccati connection $\DR$. Using equation (\ref{rc}) we can verify that 
$W_{\alpha}=g_{\alpha \beta}W_{\beta}g^{-1}_{\alpha \beta}$, i.e., $W\in H^0(M,\Lambda^2T^*M\otimes \End(TM))$.

Let $t$ be an indeterminate and $A$ an $n\times n$ matrix. Define the elementary 
symmetric polynomials $C_k$ by 
 
 \[ 
\det (tI+A)=\sum_{k=0}^{n}C_{n-k}(A)t^k 
\]

We put $R_k(M)=C_k(\frac{i}{2\pi}W_{\alpha})$, for $k=0,1,\ldots,n$. 

\begin{teo}\label{Riccaticlasses} 
Let $M$ be a complex manifold of dimension $n\geq 2$. If $M$ admits a Riccati connection, 
then 
\[ 
c_k(M)=\sum_{j=0}^{k} \binom{n-j}{k-j} R_j(M)\left(\frac{c_1(M)}{n}\right)^{k-j},
\] 
where $c_k(M)$ is the $k$-th Chern form.  
\end{teo} 
\begin{proof}  
We know that there exists a Riccati connection $\DR$, for which equation (\ref{rc}) holds.
Let $\con^0$ be a (smooth) connection on $\det(TM)$. That is, there exist (1,0)-forms 
$\eta_{\alpha}$ on $U_{\alpha}$ satisfying 
\begin{equation}\label{class2} 
\eta_{\beta}-\eta_{\alpha}=\Tr(dg_{\alpha\beta}g_{\alpha\beta}^{-1}).
\end{equation}     
Note that the connection $\con^0$ has curvature form $K_{\con^0}^{\alpha}=d\eta_{\alpha}$ and 
$c_1(M)=c_1(\det TM)=\frac{i}{2\pi}d\eta_{\alpha}$. 

Define a matrix-valued smooth (1,0)-form $\Theta_{\alpha}$ on $U_{\alpha}$ by 
\[ 
\Theta_{\alpha}=\theta_{\alpha}+\frac{1}{n}\eta_{\alpha},      
 \]     
 Using equations (\ref{rc}) and  (\ref{class2}) we have   
 \[ 
\Theta_{\alpha}=g_{\alpha\beta}\Theta_{\beta}g_{\alpha\beta}^{-1}
-dg_{\alpha\beta}g_{\alpha\beta}^{-1}
\] 
This shows that $\{\Theta_{\alpha}\}_{\alpha}$ is a (smooth) connection $\con$ 
on $TM$. Let $K_{\con}^{\alpha}$ be the form curvature of $\con$. We can see that 
 \[ 
K_{\con}^{\alpha}=W_{\alpha}+\frac{1}{n}K_{\con^0}^{\alpha}.
 \] 
 Set $A=\frac{i}{2\pi}K_{\con}^{\alpha}$, $B=\frac{i}{2\pi}W^{\alpha}$ and 
 $ \eta=\frac{i}{2\pi}K_{\con^0}^{\alpha}$. Then 
 
 \begin{align*}
\det(tI+A)&=\sum_{j=0}^n C_{n-j}(A)t^j=\sum_{j=0}^n C_{n-j}(M)t^j,
\end{align*}  
The other side we can rewrite  
 \begin{align*}
\det(tI+A)&=\det(tI+B+\frac{1}{n} \eta I)=\det \left((t+\frac{1}{n}\eta)I+B \right),\\
&=\sum_{l=0}^n C_{n-l}(B)(t+ \frac{1}{n} \eta)^l=\sum_{l=0}^n R_{n-l}(M)(t+ \frac{1}{n}c_1(M))^l,\\ 
&=\sum_{l=0}^n R_{n-l}(M)\sum_{s=0}^{l}\binom{l}{s}t^s \left(\frac{c_1(M)}{n}\right)^{l-s}, \\ 
&=\sum_{j=0}^n  \left(\sum_{l=j}^{n}\binom{l}{j}R_{n-l}(M)
\left(\frac{c_1(M)}{n}\right)^{l-j}\right)t^j. 
\end{align*}  

Comparing the terms of $t^j$ of these two equalities, we get 

 \[ 
C_{n-j}=\sum_{l=j}^{n}\binom{l}{j}R_{n-l}(M)\left(\frac{c_1(M)}{n}\right)^{l-j},  \quad 
 j=0, \ldots,n
 \] 
Exchanging the indices $n-j$ and $k$ we complete the proof of the theorem.
\end{proof} 

Using Theorem \ref{Riccaticlasses} we have $R_0(M)=1$, $R_1(M)=0$,
 \begin{align*}  
R_2(M)&=c_2(M)- \frac{(n-1)}{2n}c_1^2(M),\,\ldots
\end{align*}  
Thus we obtain the following corollary. 

\begin{cor}\label{} 
The $R_k(M)$ forms are d-closed. The de Rham cohomology classes of the $R_k(M)$ forms 
are real cohomology classes and independent of the choice of Riccati connections.     
\end{cor}  

\begin{prop}\label{restricsurfaces}
Let $M$ be a compact complex manifold of dimension $n\geq 2$ admitting a Riccati connection. 
Then  
\[
R_k(M)=0, \mbox{ for $k\geq n/2$.}  
\] 
If $M$ is a complex surface, then 
\[
4c_2(M)=c_1^2(M). 
\] 
If $M$ is K\"ahler, 
then the classes $R_k(M)$, $k\geq1$,  are zero and
\[ 
c_k(M)=n^{-k}\binom{n}{k} c_1^k(M). 
\] 
\end{prop} 
\begin{proof} Since $R_{k}(M)$ is a holomorphic $2k$-form,  
if $2k>n$ then $R_{k}(M)=0$. 

Since every $c_k(M)$ is represented by a real $(k,k)$-form, using induction we  
can show that $R_k(M)$ is represented by a real $(k,k)$-form. 

For $k=2n$, we put $\gamma=R_k(M)$ is represented by a real $(k,k)$-form $\eta$.       
Since $\eta$ is real and cohomologous to $\gamma$, it is cohomologous to 
$\overline{\gamma}$. Hence $\gamma$ and $\overline{\gamma}$ are cohomologous to each  
other. Then 
\[
\int_M \gamma\wedge \overline{\gamma}=\int_M \gamma\wedge \gamma=0.  
\] 
Thus, we get $\gamma=0$.

If $M$ is K\"ahler, then $R_k(M)=0$, since $R_k(M)$ is represented by a real $(k,k)$-form. 

\end{proof} 

As an immediate consequence of the proposition above and \cite[(2.2) Corollary]{Kob1}, we have: 
\begin{cor}
A compact K\"ahler manifold $M$ with $c_1(M)=0$ admits a Riccati connection if and 
only if it is covered by a complex torus. 
\end{cor}

%
\begin{obs}{\rm
In \cite{Molzon}, R. Molzon and K. Pinney Mortensen introduced the 
concept of a {\it projective connection} on M, which is a 
$\C$-bilinear map \[ \Pi:TM\times TM \rightarrow TM\]
satisfying 
\begin{enumerate} 
\item $\Pi_{fX}Y=f\Pi_{X}Y-\frac{1}{n+1}X(f)Y$, for $f \in \Ol(M)$,
\item $\Pi_{X}fY=f\Pi_{X}Y+X(f)Y-\frac{1}{n+1}Y(f)X$, for $f\in\Ol(M)$, 
\item $\Pi_{X}Y-\Pi_{Y}X=[X,Y]$.
\end{enumerate} 

%

In the next section, we will see that a Riccati connection on a complex surface $M$ induces a projective connection on $M$.
}
\end{obs}

From now on, unless stated otherwise,  
$M$ will denote an arbitrary complex surface (i.e., not necessarily compact).
 
\section{Riccati distributions}\label{section3}

Consider the total space $S=\Proj(TM)$ of the $\Proj^1$-bundle
$\pi:\Proj(TM)\rightarrow M$. The three dimensional variety $S$ is called  
{\it the contact variety}. 

For each point $q = (p, [v])\in S$, i.e. $p\in M$ and $v\in T_p M$, one has the plane
$\Do_q := (d\pi (q))^{-1}(\C v)$. We obtain in this way a two dimensional distribution
$\Do$ on $S$, namely the so called {\it contact distribution}.

A codimension one holomorphic regular distribution $\Ho$ on $S$ is called a  
{\it Riccati distribution} if every fibre of $\pi$ is transverse of $\Ho$.    

In local coordinates $(x, y) : U\to \C^2$, vectors $z_1\partial x + z_2\partial y$ in the fiber of $TM$ are replaced by homogeneous coordinates $(z_1 : z_2) = (1 : z)$ with $z\in\Proj^1$ in $\Proj(TM)$. Therefore, the 
$\Proj^1$-bundle writes
\begin{align*}
\pi: \Proj^1\times U&\to U \\
((1:z), (x,y)) &\mapsto (x,y),
\end{align*}
over the set $U$, 
we know that $\Do$ is given by $dy-zdx$ and the contact structure $dy = zdx$, 
and hence the Riccati distribution $\Ho$ is given by a 1-form of the type:
\begin{equation}\label{formric} 
\omega=dz+\gamma+\delta z+\eta z^2,
\end{equation}
with $\gamma$, $\delta$, and $\eta$ 1-forms on U. 

The corresponding differential equation $dz = -\gamma-\delta z-\eta z^2$
is called the {\it Riccati equation}.

We write 
\[\theta=\left(\begin{array}{rr}
-\frac{\delta}{2} &-\eta\\ 
 \gamma &\frac{\delta}{2}
\end{array} \right).\]

This matrix of 1-form represents a reduced Riccati connection $\DR$. 
In fact, taking two local systems of coordinates $(\varphi_{\alpha} ,U_{\alpha})$ and $(\varphi_{\beta},U_{\beta})$ on $M$ with $U_{\alpha}\cap U_{\beta}\neq \emptyset$ and letting 
$\varphi_{\alpha\beta}=\varphi_{\alpha}\circ \varphi_{\beta}^{-1}$ denote the corresponding change of coordinates, we have
\begin{equation} \label{cocycleric} 
\omega_{\alpha}=h_{\alpha\beta} \omega_{\beta},  \quad h_{\alpha\beta}\in\Ol^*(\pi^{-1}U_{\alpha}
\cap \pi^{-1}U_{\beta}), 
\end{equation}
and $\Ho_{|_{\pi^{-1}U_{\alpha}}}=\{\omega_{\alpha}=0\}$.  
Using equation (\ref{cocycleric}) we get 
 
\begin{equation} \label{equationdistribution}
\theta_{\alpha}=g_{\alpha\beta}\theta_{\beta}g_{\alpha\beta}^{-1}
-dg_{\alpha\beta}g_{\alpha\beta}^{-1}+\frac{1}{2}\Tr(dg_{\alpha\beta}g_{\alpha\beta}^{-1})I. 
\end{equation} 
Furthermore, the following conditions are equivalent 
\begin{itemize}
\item $\DR$ has zero curvature: $\theta\wedge \theta+d\theta=0$; 
\item $\Ho$ is integrable (Frobenius): $\omega\wedge d\omega=0$.  
\end{itemize}
Therefore, we obtain the following. 

\begin{prop}\label{conectionsdistrbutions} 
We have a one-to-one correspondence between reduced 
Riccati connections on $M$ and 
Riccati distributions on $\Proj(TM)$. 

Furthermore, this correspondence induces  
a one-to-one correspondence between reduced 
Riccati connections with zero curvature on $M$ and   
Riccati foliations on $\Proj(TM)$, i.e. Frobenius integrable Riccati distributions. 

In particular, finding Riccati foliations on $\Proj(TM)$ is equivalent to finding 
$2\times 2$ holomorphic matrices of 1-forms $\theta_{\alpha}$ on $U_{\alpha}$ 
that satisfy the structure equations 
\begin{align}
d\theta_{\alpha}+\theta_{\alpha}\wedge \theta_{\alpha}&=0, \label{curvature}  \\
\Tr(\theta_{\alpha})&=0,\label{trace}
\end{align}
and whose changes of coordinates satisfy (\ref{equationdistribution}). 
\end{prop}

\begin{cor}\label{projectivecon}
If $M$ admits a Riccati connection, then it admits a projective connection.  
\end{cor}
\begin{proof}
By Proposition \ref{conectionsdistrbutions} there is $\Ho$ a Riccati distribution on $\Proj(TM)$. 
By intersecting $\Ho$ with the contact distribution $\Do$ we obtain a geodesic foliation $\Go$ that induces 
a projective connection on $M$. See \cite[Section 2]{FallaLoray} for more details.  
\end{proof}


\subsection{Curvature of a Riccati distribution}
Using the same notation as in (\ref{formric}),
we write 

\[
\left(\begin{array}{c}
\gamma\\
\delta\\
\eta
\end{array}\right)
=\left(\begin{array}{c}
\gamma_1\\
\delta_1\\
\eta_1
\end{array}\right)dx+
\left(\begin{array}{c}
\gamma_2\\
\delta_2\\
\eta_2
\end{array}\right)dy, 
\quad \gamma_{i}, \delta_i,\eta_i\in \Ol(U). 
\]

We define the following 1-form 
 \[   
\kappa=(  \frac{\delta_{1}}{2}- \gamma_{2})dx+( \eta_{1}-  \frac{\delta_{2}}{2})dy, 
\] 
which we call 
the {\it connection form}. Now, we will see that $\kappa$  
determines the so-called holomorphic connection on $\det(TM)$. 

\begin{prop}\label{connecdet} 
If $((x,y),U_{\alpha})$, $((\tilde x,\tilde y),U_{\beta})$ are two local systems of coordinates on $M$,  
with $U_{\alpha}\cap U_{\beta}\neq \emptyset$, then 
\[
\kappa_{\alpha}-\kappa_{\beta}=\frac{1}{2}\Tr(dg_{\alpha\beta}g^{-1}_{\alpha\beta})
=\frac{1}{2}d(\log (\det g_{\alpha\beta})). 
\]
In particular, $\{-2\kappa_{\alpha}\}_{\alpha}$ defines a holomorphic connection on 
the cano\-nical bundle $K_M$.   
\end{prop}

\begin{proof}
We denote by 
$\kappa=\kappa_{\alpha}=\kappa_1dx+\kappa_2 dy$, 
$\tilde{\kappa}=\kappa_{\beta}=\tilde\kappa_1d\tilde x+\tilde\kappa_2 d\tilde y$ and 
$g=g_{\alpha \beta}=\left(\begin{array}{cc}
g_{11}&g_{12}\\
g_{21}&g_{22}
\end{array}\right)$. Let $\DR$ be the Riccati 
connection induced by $\Ho$, and we define 
$T(X,Y)=\DR_{X}Y-\DR_{Y}X-[X,Y]$. 
We see that
$T(\partial_x,\partial_y)=-\kappa_2\partial_x+\kappa_1\partial_y$.
 
We can verify the following properties: $T$ is $\C$-bilinear 
and
\[
T(fX,gY)=fgT(X,Y)+\frac{1}{2}(fY(g)X-gX(f)Y),
\]
where $f,g\in\Ol_M$, $X,Y\in TM$.  
Using these properties and the fact that $\partial_{\tilde x}=g_{11}\partial_x+g_{21}\partial_y$ and  
$\partial_{\tilde y}=g_{12}\partial_x+g_{22}\partial_y$, we get  

\begin{align}\label{E:connection1}
T(\partial_{\tilde x},\partial_{\tilde y})
&=(-(\det g) \kappa_2+\frac{1}{2}h_1)\partial_x
+((\det g )\kappa_1+\frac{1}{2}h_2)\partial_y,
\end{align}
where, 
\begin{align*}
h_1&=g_{11}(\partial_x(g_{12})+\partial_y(g_{22}))-g_{12}(\partial_x(g_{11})+\partial_y(g_{21})),\\
h_2&=g_{21}(\partial_x(g_{12})+\partial_y(g_{22}))-g_{22}(\partial_x(g_{11})+\partial_y(g_{21})).
\end{align*}

On the other hand, we have 
\begin{align}\label{E:connection2}
T(\partial_{\tilde x},\partial_{\tilde y}) 
&=(-\tilde \kappa_2 g_{11}+\tilde \kappa_1 g_{12})\partial_x+
(-\tilde \kappa_2 g_{21}+\tilde \kappa_1 g_{22})\partial_y.
\end{align}
Comparing (\ref{E:connection1}) and (\ref{E:connection2}), we get 

\[
(\det g) \kappa=(\det g) \tilde \kappa+\frac{1}{2}((h_1g_{21}-h_2g_{11})d\tilde x+(h_1g_{22}-h_2g_{12})d\tilde y).
\]
It is not difficult to verify that  
\[
(\det g) \Tr(dg.g^{-1})=(h_1g_{21}-h_2g_{11})d\tilde x+(h_1g_{22}-h_2g_{12})d\tilde y. 
\]
This completes the proof of the proposition.
\end{proof}

We define the {\it curvature} of $\Ho$ as 
\[   
K({\Ho})=d\kappa\in \Omega^2(M), 
\] 
We say that $\Ho$ is {\it parallelizable} if $K({\Ho})=0$. 

\begin{teo} \label{flat}
Let $\Ho$ be a Riccati distribution on $\Proj(TM)$ and $\{-2\kappa_{\alpha}\}_{\alpha}$ be the
holomorphic connection on the canonical bundle $K_M$ induced by $\Ho$. Then 
$\{-2\kappa_{\alpha}\}_{\alpha}$ is flat if and only if $\Ho$ is parallelizable. 

If $M$ is a compact complex surface,  
then $\Ho$ is parallelizable.
\end{teo}

\begin{proof}
The first part follows directly from the definition. 

For the second part, by Proposition \ref{connecdet}, the 1-forms $\{2\kappa_{\alpha}\}_{\alpha}$ define a holomorphic connection 
on $\det (TM)$. The curvature of this connection is $2d\kappa$, which represents $c_1(\det TM)=c_1(M)$. 
We conclude by following the same ideas as Proposition \ref{restricsurfaces}. 
\end{proof}

\subsection{Riccati distributions and affine connections}

An affine connection on $M$ is a (linear) holomorphic connection on the tangent bundle 
$T M$, i.e., a $\C$-bilinear map $\con:TM\times TM\rightarrow TM$ 
satisfying the Leibnitz rule $\con_X(f \cdot Z)= f \cdot\con_X(Z)+df(X) Z$ 
and $\con_{fX}( Z)=f\con_X(Z)$, 
for any holomorphic function $f$ and any vector fields $X,Z$.
The connection $\con$ is torsion free when the torsion vanishes, that is 
$\con_XZ-\con_X Z-[X,Z]=0$, for all vector fields $X,Z$, and the curvature of $\con$ 
is denoted by $K_{\con}=\con\cdot\con$.  
The connection $\con$ is flat when the curvature vanishes, that is $K_{\con}=\con\cdot\con=0$. 

We describe a map from the set of connections to into the set of distributions as follows:
In coordinates $(x,y)\in U\subset \C^2$, a trivialization of $TU$ is given by the basis 
$(\partial_x,\partial_y)$ and the affine connection is given by
\[\con(Z)=d(Z)+\theta Z, \quad \theta=\left(\begin{array}{cc}
\theta_{11}&\theta_{12}\\
\theta_{21}&\theta_{22}
\end{array}\right),  \]
where $Z=z_1\partial_x +z_2\partial_y$ and $\theta_{ij}\in\Omega^1(U)$. On the projectivized bundle 
$\Proj(TU)$, with trivializing coordinate $z=z_2/z_1$, the equation $\con=0$ induces a 
Riccati distribution $\Ho^{\con}$ that is locally given by 

\begin{equation}\label{affindis} 
\omega=dz+\theta_{21}+(\theta_{22}-\theta_{11})z-\theta_{12}z^2.
\end{equation}

\begin{teo} \label{Riccatiaffineconnection}
This map induces a one-to-one correspondence between
torsion free affine connections on $M$ and Riccati distributions on $\Proj(TM)$.

Furthermore, there are one-to-one correspondences between:  
\begin{enumerate} 
\item Torsion free affine connections $\nabla$ on $M$ with   
zero curvature of the trace of connection $\Tr(\nabla)$, i.e.,  $K_{\Tr(\nabla)}=0$,  
and paralleli\-za\-ble distributions on $\Proj(TM)$.

\item Torsion free affine connections $\nabla$  
on $M$ with $K_{\nabla}=\frac{1}{2}K_{\Tr(\nabla)}I$ and Riccati foliations on $\Proj(TM)$.

\item 
Affine structures on $M$ and parallelizable Riccati foliations on $\Proj(TM)$.   
\end{enumerate}
\end{teo}

\begin{proof} To verify that we have a bijection we will describe the inverse mapping 
as follows:
Let $\Ho$ be a Riccati distribution on $\Proj(TM)$. By Proposition \ref{conectionsdistrbutions} 
we have a reduced Riccati connection $\DR$ induced by $\Ho$ that verifies  
$\theta_{\alpha}=g_{\alpha\beta}\theta_{\beta}g_{\alpha\beta}^{-1}
-dg_{\alpha\beta}g_{\alpha\beta}^{-1}+\frac{1}{2}\Tr(dg_{\alpha\beta}g_{\alpha\beta}^{-1})I$, 
where $\theta_{\alpha}$ is the matrix of 1-forms of $\DR$ in $U_{\alpha}$. 
By Proposition \ref{connecdet} we have 
$\frac{1}{2}\Tr(dg_{\alpha\beta}g_{\alpha\beta}^{-1})=\kappa_{\alpha}-\kappa_{\beta}$, 
where $\kappa_{\alpha}\in \Omega^1(U_{\alpha})$ is the connection form. Define 
$\tilde\theta_{\alpha}=\theta_{\alpha}-\kappa_{\alpha}I$,  so that 
$\tilde\theta_{\alpha}=g_{\alpha\beta}\tilde\theta_{\beta}g_{\alpha\beta}^{-1}
-dg_{\alpha\beta}g_{\alpha\beta}^{-1}$. 
Then the 1-forms $\{\tilde\theta_{\alpha}\}_{\alpha}$ define an affine connection
 $\nabla^{\Ho}$ on $M$.
  
To verify items 1., 2. and 3. it is sufficient to see that $K_{\Tr(\nabla)}=-2K(\Ho^{\con})$ 
and $K_{\con}=W-K(\Ho^{\con})I$, where $W$ is the curvature of the Riccati connection 
induced by $\Ho^{\con}$. The theorem is proved.    
\end{proof} 

\begin{cor}\label{afinRicadistri} 
The following asser\-tions are equivalent
\begin{itemize}
\item $M$ admits an affine connection,
\item $M$ admits a Riccati connection,
\item $\Proj (TM)$ admits a Riccati distribution.   
\end{itemize}
\end{cor}

It is worth pointing out that  by using Corollary \ref{projectivecon} 
together with the classification of (normal) projective connections 
\cite{Kob2, Kob3} and Proposition \ref{restricsurfaces}  one can obtain another proof of Corollary  \ref{afinRicadistri} in the case of a compact complex surface.  

\begin{cor} 
If $M$ is compact, then we have a one-to-one corres\-pondence between 
Riccati foliations on $\Proj(TM)$ and affine structures on $M$.  
\end{cor}

\begin{proof} 
This follows from Theorem \ref{flat} and Theorem \ref{Riccatiaffineconnection}.   
\end{proof} 

\subsection{Monodromy of a Riccati foliation}

The {\it monodromy representation} of the Riccati foliation $\Ho$ on $S=\Proj(TM)$ is the representation
\[
\rho_{\Ho}: \pi_1(M)\rightarrow \aut(\Proj^1)
\]
defined by lifting paths on $M$ to the leaves of $\Ho$. 
The image of $\rho_{\Ho}$ in $\aut(\Proj^1)$ is, by definition, the {\it monodromy group} of $\Ho$, 
denoted by $\mon(\Ho)$. Let's see some particular cases. 

\subsubsection{Pencils of foliations and Riccati foliations} 

A regular pencil of foliations $\Po$ on $U$ is a one-parameter family of foliations 
$ \{ \Fo_t \}_{t \in \Proj^1}$ defined by $ \Fo_t =[ \omega_t = 0]$ for a pencil of 1-forms 
$\{ \omega_t= \omega_0+t \omega_{ \infty}\}_{t \in \Proj^1}$ with 
$\omega_0, \omega_{ \infty} \in \Omega^1(U)$ and $ \omega_0 \wedge  \omega_{ \infty}\neq 0$ 
on $U$. The pencil of 1-forms defining $ \{ \Fo_t \}_{t \in \Proj^1}$ is unique up to multiplication by
a non vanishing function: $\tilde\omega_t=f\omega_t$ for all $t \in \Proj^1$ and  $f\in \Ol( U)$. In fact, the parametrization by $t \in \Proj^1$ is not intrinsic; we will say that $ \{ \Fo_t \}_{t \in \Proj^1}$ and $ \{  \tilde\Fo_t \}_{t \in \Proj^1}$ define the same pencil on $U$ if there is a M\"obius transformation $\varphi \in \aut(\Proj^1)$ such that $ \tilde\Fo_t= \Fo_{\varphi(t)}$ for all 
$t \in \Proj^1$. \\
The graphs $S_t$ of the foliations $\Fo_t$ are disjoint sections (since foliations 
are pairwise transversal) and form a codimension 
one foliation $ \Ho$ on $ \Proj(TU)$ transversal to the projection $ \pi: \Proj(TU) \rightarrow U$. The foliation $ \Ho$ is a Riccati foliation, i.e. a Frobenius integrable Riccati distribution:
\[   
\Ho:[\omega=0], \quad \omega=dz+\gamma+\delta z+\eta z^2, \quad \omega\wedge d\omega=0.
\] 

In local coordinates $(x, y)$ such that $\Fo_0$ and $\Fo_{\infty}$ are defined by 
$dx = 0$ and $dy = 0$ respectively, we can assume the pencil is generated by $\omega_0=dx$ and 
$\omega_{\infty}=u(x, y)dy$ (we have normalized $\omega_0$) with $u(0, 0)\neq0$. Then the 
graph of each foliation $\Fo_t$ is given by the section $S_t = \{z=-\frac{1}{tu(x,y)}\}\subset \Proj(TU)$. 
These sections are the leaves of the Riccati foliation $\Ho : [dz +\frac{du}{u}z = 0]$.

We know that the curvature $K(\Po)$ of a regular pencil $\Po=\{\Fo_t\}_{t\in\Proj^1}$ 
is a $2$-form (see \cite{Alcides}). 
For instance, if $\Po$ is in normal form $\Po = dx + tu(x, y)dy$, then the curvature is given by
\[
K(\Po) = -(\ln u)_{xy} dx\wedge dy.
\]
We recall that a pencil $\Po$ is {\it flat} when its curvature is zero, that is $K(\Po)=0$.

On the other hand, $\Po$ induces a Riccati foliation $\Ho : [dz +\frac{du}{u}z = 0]$ on $\Proj(TU)$.
So we have $\kappa=\frac{u_x}{2u}dx-\frac{u_y}{2u}dy$. Thus 
the curvature of $K(\Ho)=K(\Po)=d\kappa$, this means that the definition of the curvature of a 
Riccati distribution extends the definition of the curvature of a pencil of foliations. 
So, we have the following proposition. 
\begin{prop} \label{pencilRiccati} 
The following data are equivalent:
\begin{itemize}
\item a regular pencil of foliations $\{\Fo_t\}_{t\in\Proj^1}$ on $M$,
\item a Riccati foliation $\Ho$ on $\Proj(TM)$ with trivial monodromy, i.e., $\mon(\Ho)=\{id\}$. 
\end{itemize}
Furthermore, if $M$ is compact and $\Po$ is a regular pencil on $M$, then $\Po$ is flat. 
\end{prop}
\begin{proof} 
It follows from the previous construction and the fact that the leaves of the foliation $\Ho$ are defined 
by the graphs $S_t$ of the foliations $\Fo_t$ that the monodromy is trivial.
\end{proof}

The following lemma exhibits the normal form of a flat pencil that represents an affine structure. 

\begin{lem} (\cite[Lemma 2.1.4]{Alcides}). 
Let $\Po$ be a flat regular pencil defined in a neighborhood of
the origin $0\in\C^2$. Then, there is a change of local coordinates 
$\varphi(x, y)$ sending $\Po$ to the pencil defined by $dx + tdy$, $t\in \Proj^1$.

\end{lem}

\subsubsection{Webs and pencils of foliations}

Let $\W= \Fo_1\boxtimes  \Fo_2 \boxtimes \Fo_3 \boxtimes \Fo_4$ be a regular 4-web on $(\C^2, 0)$
\[
\Fo_i = [X_i = \partial x + e_i(x, y)\partial y] = [\eta_i = e_idx-dy], \,\,\,i=1, 2, 3, 4.
\]
The cross-ratio
\[
 (\Fo_1, \Fo_2, \Fo_3, \Fo_4) := \frac{(e_1-e_3)(e_2-e_4)}{(e_2-e_3)(e_1-e_4)}
\]
is a holomorphic function on $(\C^2, 0)$ intrinsically defined by $\W$. Then, we have:
\begin{prop}
If $\W = \Fo_0\boxtimes \Fo_1\boxtimes \Fo_{\infty}$ is a regular 3-web on $(\C^2, 0)$, then there is
a unique pencil $\{\Fo_t\}_{t\in \Proj^1}$ that contains $\Fo_0, \Fo_1$ and $\Fo_{\infty}$ as 
its elements. More precisely, $\Fo_t$ is the only foliation such that
\[
(\Fo_t, \Fo_0, \Fo_1, \Fo_{\infty})=t.
\]
\end{prop}

Conversely, any Riccati foliation comes from a 3-web: it suffices to choose
3 elements of a pencil. In particular, any 4 elements of a pencil $\{\Fo_t\}_{t\in \Proj^1}$ have
constant cross-ratio.

In general, we can define a cross-ratio for a $k$-web, $k>3$, on a compact complex surface $M$.
Indeed, in a finite cover of $M$ any four foliations define a constant cross-ratio, and therefore they define a cross-ratio on $M$, thus giving a Riccati foliation on $M$.

The proposition below follows from the previous considerations.

\begin{prop}\label{kwebs}
If $M$ is compact and $\W$ is a regular $k$-web on $M$, $k\geq 3$, then there is a Riccati foliation $\Ho$ 
on $\Proj(TM)$ induced by $\W$ such that 
\[
\mon(\Ho)\subset \sym (1, \ldots, k). 
\]
\end{prop}

See \cite{Pereira} and \cite{FallaLoray} for more details. 

\subsection{Calculations of the monodromy of a Ricca\-ti foliation}\label{calculations}
 What we will do is the following: calculate the Riccati connections on $M$ of zero curvature 
 and zero trace, and hence the space of Riccati foliations on $\Proj(TM)$.
If $M$ has at least one affine structure, the Riccati connections  
are the elements of $H^0(M,V)$, where $V=End(TM)\otimes T^*M$ 
(vector bundle, Higgs field)=:Higgs bundle. 
This follows from (\ref{equationdistribution}) since $g_{\alpha\beta}$, 
can be chosen constant so that $dg_{\alpha\beta}= 0$. 
Now pull back $V$ via the universal cover $\pi:\tilde M \rightarrow M$ of $M$, and compute 
the holomorphic sections of $\pi^*(V)$; those invariant by $\pi_1(M)$ will be the holomorphic 
sections of $V$. 

The universal cover $\tilde M$ of each of these surfaces is a subdomain of $\C^2$ and the
covering transformations are affine so that the standard flat holomorphic
linear connection on $\C^2$ restricted to $\tilde M\subset \C^2$ can be "pulled down"
to $M$. In this case, $M$ has a natural or usual affine
connection whose corresponding affine coordinates are those of $\tilde M\subset \C^2$
defined locally on $M$. Since $\pi^*(V)=\tilde M\times \C^8$, the Riccati connections 
on $\tilde M$ are of the form $\theta= A_1(x,y)dx +A_2(x,y)dy$ where $A_k(x,y)$ is a 
holomorphic $2\times 2$ matrix and $(x,y)$ are the global coordinates on 
$\tilde M\subset \C^2$. The Riccati connections on $M$ are the 2-forms $\theta$ that are invariant 
by $\pi_1(M)$. Using the structure equations (\ref{curvature}) and (\ref{trace}), the zero curvature condition becomes
\[
0=\frac{\partial A_2 }{\partial x}-\frac{\partial A_1}{\partial y}+[A_1,A_2], 
\]
and the relation $\Tr(\theta)=0$ turns into $\Tr(A_1)=\Tr(A_2)=0$.

\subsubsection{Complex torus surfaces}
In this case $M$ is the quotient of $\C^2$ by some lattice 
$\Gamma=\oplus_{i=1}^{4}(k_i,l_i)\Z$. Then  
$\theta= A_1dx +A_2dy$, where $A_1,A_2$ are complex $2\times 2$ matrices such that
\begin{equation}\label{conjugation}
A_1A_2=A_2A_1 \mbox{ and $\Tr(A_1)=\Tr(A_2)=0$.}
\end{equation}

The pairs of matrices $(A_1, A_2)$ are divided into conjugacy classes: $(A_1, A_2)$ is said to
be conjugate to $(B_1,B_2)$ iff there exists $C\in\gl(2,\C)$ such that 
$CB_1C^{-1}=A_1$ and $CB_2C^{-1}= A_2$.
Clearly, if we can find one member $(B_1,B_2)$ of a conjugacy class satisfying (\ref{conjugation}) then 
the other members $(A_1,A_2)$ are easily obtained. So we may assume that $B_1$, $B_2$ are in 
Jordan normal form, and in this case what we obtain is given in the listed below:
\[
\begin{tabular}{|l|l|l|l|}
\cline{1-4}
Type & $B_1=$ & $B_2=$  & Monodromy \\
\cline{1-4}
1 &  $\left(\begin{array}{cc}
-\frac{a}{2}&0\\
0&\frac{a}{2}
\end{array} \right)$  &
$\left(\begin{array}{cc}
-\frac{b}{2}&0\\
0&\frac{b}{2}
\end{array}\right)$  & $z\mapsto z\exp(ak_i+bl_i)$   \\
\cline{1-4}
2 & $\left(\begin{array}{cc}
0&1\\
0&0
\end{array} \right)$ &  
$\left(\begin{array}{cc}
0&c\\
0&0
\end{array}\right)$   &  
$z\mapsto z-(k_i+cl_i)$   \\
\cline{1-4}
3 &$\left(\begin{array}{cc}
0&0\\
0&0
\end{array} \right)$ &  
$\left(\begin{array}{cc}
0&1\\
0&0
\end{array}\right)$  & $z\mapsto z-l_i$ \\
\cline{1-4}  
\end{tabular}
\]
where $a,b,c\in \C$, $i=1,2,3,4$ and the corresponding monodromy 
(seen as a M\"obius transformation on one complex variable $z$) is also included.

We write $C=\left(\begin{array}{cc}
e&f\\
g&h
\end{array} \right)$. Then, in general, for $A_1$ and $A_2$ we have:

\[\tiny
\begin{tabular}{|l|l|l|}
\cline{1-3}
Type &  Riccati foliation $\Ho$ induced by $\theta$  & Monodromy \\
\cline{1-3}
1 &  $dz+\frac{1}{\det C}(g-ez)(-h+fz)(adx+bdy)$  & $z\mapsto z\exp(ak_i+bl_i)$   \\
\cline{1-3}
2 & $dz+\frac{1}{\det C}(g-ez)(-g+ez)(dx+c dy)$ &  
$z\mapsto \left\{\begin{array}{lc}
 z-\frac{g}{f}(k_i+cl_i), &\mbox{ if $e=0$},\\ 
 z-e(k_i+cl_i), & \mbox{ if $e\neq 0$}. 
\end{array}
\right.$\\
\cline{1-3}
3 & $dz+\frac{1}{\det C}(g-ez)(-g+ez)dy$  & 
$z\mapsto \left\{\begin{array}{lc}
z-\frac{g}{f}l_i,&\mbox{ if $e=0$},\\ 
z-el_i,& \mbox{ if $e\neq 0$}. 
\end{array}
\right.$   \\
\cline{1-3}
\end{tabular}
\]
where $a,b,c\in \C$ and $i=1,2,3,4$. 

From the list above we deduce the following lemma.

\begin{lem} \label{L:torustrivial}
Let $\Ho$ be a Riccati foliation on $\Proj(T(\C^2/\Gamma))$ induced by $\theta$. 
Then $\mon(\Ho)$ is trivial if and only if $\theta=0$, i.e., it is of type 1 with $a=b=0$. 
Furthermore, if $\mon(\Ho)$ is finite, then $\mon(\Ho)$ is trivial. 
\end{lem}

\begin{proof} 
It is enough to verify that if the monodromy is finite, then $\theta=0$. 
In fact, we can assume that $(k_3,l_3)=(1,0)$, $(k_4,l_4)=(0,1)$ and 
$\det(\im \tau)\neq 0$, where $\tau=\left(\begin{array}{cc}
k_1&k_2\\
l_1&l_2
\end{array}\right)$. 
So, $\Ho$ is of type 1 and $\mon(\Ho)$ is generated by $h_j(z)=z\exp(ak_j+bl_j)$ such that 
$h_j^n(z)=z$, for all $j=1,2,3,4$ and for some $n\in\N$. Thus we see $na=2\pi i r$ and $nb=2\pi is$, $r,s\in\Z$.  
Also, we have $rk_1+sl_1\in\Z$ and $rk_2+sl_2\in\Z$. Looking at the imaginary part of these 
numbers we deduce $\im\tau^{T}\left(\begin{array}{c}
r\\
s
\end{array}\right)=\left(\begin{array}{c}
0\\
0
\end{array}\right)$, so we get $r=s=0$ and therefore $a=b=0$.   
\end{proof}

\subsubsection{Hyperelliptic surfaces} 
A hyperelliptic surface, or bi-elliptic surface, is a surface that can be written as the quotient of a product 
of two elliptic curves by a finite abelian group. 
According to \cite[Theorem 4]{Bombieri}, there are seven types of hypereliptic surfaces that can written  
as $M=E\times F/G$, where $E$, $F$ elliptic curves and $G$ is finite abelian subgroup of $E$. Therefore 
we can also calculate the monodromy of Riccati foliations on $\Proj (TM)$, and these are given in the 
following table.
 \[
\begin{tabular}{|l|l|l|}
\cline{1-3}
$G$ & Action of $G$ &  Monodromy \\
\cline{1-3}
$\Z_2$ &  $(x,y)\to(x+a,-y)$& $z\mapsto -z$  \\
\cline{1-3}
$\Z_2\oplus\Z_2$ &  $(x,y)\to(x+a,-y)$, $(x+b,y+c)$, $c=-c$& $z\mapsto -z$ \\  
\cline{1-3}
$\Z_3$ &  $(x,y)\to(x+a,\nu y)$  & $z\mapsto \nu z$\\
\cline{1-3}
$\Z_3\oplus \Z_3$ &  $(x,y)\to(x+a,\nu y)$, $(x+b,x+c)$, $\nu c=c$& $z\mapsto \nu z$\\
\cline{1-3}
$\Z_4$ &  $(x,y)\to(x+a,iy)$  & $z\mapsto iz$\\
\cline{1-3}
$\Z_2\oplus \Z_4$ &  $(x,y)\to(x+a,iy)$, $(x+b,y+c)$, $ic=c$&  $z\mapsto iz$ \\
\cline{1-3}
$\Z_6$ &  $(x,y)\to(x+a,-\nu y)$ & $z\mapsto-\nu z$ \\
\cline{1-3}  
\end{tabular}
\] 
where $\nu$ is a primitive cube root of 1. The following lemma is a straightforward consequence 
of the above table.      
\begin{lem}\label{L:hyperfinite}
Let $\Ho$ be a Riccati foliation on $\Proj (T(E\times F/G))$. Then 
$\mon(\Ho)$ is non-trivial, finite and cyclic. 
In particular, there are no regular pencils on $E\times F/G$ and there exists regular $k$-webs on $E\times F/G$, 
for $k\geq 3$.
\end{lem}

\subsubsection{Primary Kodaira surfaces}
A primary Kodaira surface $K$ is the quotient of $\C^2$ by some group $G$ generated by
\begin{align*}
g_1(x, y) = (x,y+1),  & &g_3(x, y) = (x+1,ax+y), \\
 g_2(x, y) = (x,y+\tau_1), & &g_4(x, y) = (x+\tau_2,bx+y), 
\end{align*}
where $\tau_1,\tau_2$ are generators of the fundamental domain of the modular group and $a,b\in \C$ satisfy $a\tau_2-b=m\tau_1$ for some positive integer $m$. See more details in \cite{Vitter}.

A Riccati connection on $K$ has the form
\[
\theta= \left(\begin{array}{cc}
e&0\\
c&h
\end{array} \right)dx+\left(\begin{array}{cc}
0&0\\
e-h&0
\end{array} \right)dy,
\]
for some constants $e$, $c$, $h$. Therefore the Riccati foliations on $\Proj(TK)$ 
are of the form $dz+cdx$, $c\in \C$. 
So, the monodromy group $\mon$ is generated by $f_1(z) =z-a-c$, and  
$f_2(z) = z-b-c\tau_2$. 
From this we can conclude that there are no regular pencils $\Po$ and  regular 
$k$-webs $\W$ on a primary Kodaira surface, for $k\geq 3$.

\subsubsection{Hopf surfaces}
A compact complex surface $M$ whose universal covering space is biholomorphic to    
$\C^2-\{0\}$ is called a Hopf surface.

{\it 1. Primary Hopf surfaces}. We consider first the primary Hopf surfaces, which are quotients of 
$\C^2-\{0\}$ by the infinite
cyclic group generated by an automorphism $g$ of $\C^2-\{0\}$. According to 
\cite[Part II]{Kodaira}, § 10, $g$ has the form:
\[
g(x, y) = (ax+\lambda y^m, by),  
\]
for some positive integer $m$ and some complex numbers $a$,$b$, $\lambda$ 
with $0<|a|\leq |b|< 1$ and $(a-b^m)\lambda=0$. 

By using \cite[(7.5) Theorem]{Kob1} we have:
\[
\begin{tabular}{|l|l|l|}
\cline{1-3}
Condition & Riccati foliation $\Ho$ induced by $\theta$  & Monodromy\\ 
\cline{1-3}
 $\lambda\neq 0$, $m=1$ &  $dz$  & $z\mapsto \frac{zb+\lambda}{a}$   \\
 \cline{1-3}
$\lambda=0$, $a\neq b^2$ & $dz$ &  
$z\mapsto \frac{zb}{a}$ \\
\cline{1-3}
$\lambda=0$, $a=b^2$ & $dz-cz^2dy$  & 
$z\mapsto \frac{z}{b}$  \\
\cline{1-3}
\end{tabular}
\]
where $c\in \C$.  

From the above list we deduce directly the following.

\begin{lem} Let $M$ be a primary Hopf surface and $\Ho$ be a Riccati foliation on $\Proj(TM)$. 
\begin{enumerate}
\item  $\mon(\Ho)$ is trivial if and only if $\theta=0$,i.e., $\lambda=0$ and $a=b$ . 
\item $\Ho$ is induced by a regular $k$-web, $k\geq 3$, if and only if $\lambda=0$ and 
$a=\xi b$, where $\xi$ is a primitive $j$-th root of 1, for some integer $j\geq 1 $. 
\end{enumerate}
\end{lem}

{\it 2. Secondary Hopf surfaces}. A secondary Hopf surface $M$ is the quotient of $\C^2-\{0\}$ 
by the free action of a group $\Gamma$ containing a central, finite index subgroup generated by an 
automorphism $g$ of the above type. The primary Hopf surface $N = \C^2-\{0\}/g$ is 
a finite \'etale cover of $M$ and the corresponding finite subgroup is generated by 
$e(x, y) = (\epsilon_1x, \epsilon_2y)$, 
where $\epsilon_1$ and $\epsilon_2$ are primitive $l$-th roots of unity and 
$(\epsilon_1-\epsilon_2^m)\lambda=0$. Then we have:
\[
\begin{tabular}{|l|l|l|}
\cline{1-3}
Condition & Riccati foliation $\Ho$ induced by $\theta$  & Monodromy \\
\cline{1-3}
 $\lambda\neq 0$, $m=1$ &  $dz$  & $z\mapsto \frac{z\epsilon_2}{\epsilon_1}$   \\
\cline{1-3}
$\lambda=0$, $\epsilon_1\neq \epsilon_2^2$ & $dz$ &  $z\mapsto \frac{z\epsilon_2}{\epsilon_1}$  \\
\cline{1-3}
$\lambda=0$, $\epsilon_1=\epsilon_2^2$ & $dz-cz^2dy$  & $z\mapsto \frac{z}{\epsilon_2}$  \\
\cline{1-3}
\end{tabular}
\]
where $c\in \C$. 

Hopf surfaces with fundamental group isomorphic to $\Z\oplus \Z / l\Z $, where $l\geq1 $ is an integer, are 
generated by two diagonal automorphisms $g(x, y)=(ax, ay) $ with $0 <|a|<1$, and 
$e( x, y)=(\epsilon x, \epsilon y)$ where $\epsilon$ is a primitive $l$-th root of 1. We denote them by $H_{a, l}$. 
From these two previous lists we conclude the following.

\begin{lem} \label{L:hopffinite}
Let $M$ be a Hopf surface and $\Ho$ be a Riccati foliation on $\Proj(TM)$. 
\begin{enumerate}
\item  $\mon(\Ho)$ is trivial if and only if $M=H_{a,l}$. 
\item $\Ho$ is induced by a regular $k$-web, $k\geq 3$, if and only if $H_{a,1}$ is a finite cover of $M$.
\item If $\mon (\Ho)$ is finite, then $\mon(\Ho)$  is cyclic.
\end{enumerate}
\end{lem}

\subsubsection{Inoue surfaces}

Inoue surfaces are compact complex surfaces of type $VII_0$, see \cite{BPV} for more details. 
In \cite{Inoue} Inoue shows that these surfaces are obtained as quotients of $\Hip\times \C$ by a group 
$\Gamma$ of affine transformations of $\C^2$ preserving the open set $\Hip$ ($\Hip$ being the Poincar\'e upper half-plane). In particular, each Inoue surface inherits an affine structure induced by the canonical 
affine structure of $\C^2$, which is unique by \cite[Lemma 4.3]{Klingler}. Up to a double 
unramified cover, Inoue surfaces are obtained by one of the following two procedures \cite{Inoue}.

1. {\it Surfaces $S_M$}. Consider a matrix $M\in \SL(3,Z)$ with eigenvalues 
$\alpha, \beta, \overline{\beta}$ such that $\alpha> 1$ and $\beta\neq \overline{\beta}$. 
Choose a real eigenvector $(a_1, a_2, a_3)$  associated with the eigenvalue $\alpha$ and an eigenvector 
$(b_1, b_2, b_3)$ associated with the eigenvalue $\beta$. Consider also the group $\Gamma$ of (affines) transformations of $\C^2$ generated by:

\begin{align*}
\gamma_0(x ,y) &= (\alpha x, \beta y),  \\
\gamma_i(x, y) &= (x + a_i, y + b_i), \mbox{ with $i = 1, 2, 3$.}
\end{align*}

The action of $\Gamma$ on $\C^2$ preserves $\Hip \times\C$ and the quotient is the compact  complex surface $S_M$.

In this case, there is only one Riccati foliation on $\Proj(TS_M)$  and it has the form 
$dz$. Thus, the monodromy group is generated by $f(z) =\frac{\beta}{\alpha}z$.

2. {\it Surfaces $S^+_{N,p,q,r,t}$}. Let $N = (n_{ij} )\in\SL(2,\Z)$ be a diagonalizable matrix
on $\R$ with eigenvalues $\alpha > 1$ and $\alpha^{-1}$ 
and eigenvectors $(a_1, a_2)$ and $(b_1, b_2)$ 
respectively. Choose $r\in\Z^*$, 
$p$, $q\in \Z$, $t\in\C$ and real solutions $c_1,c_2$ of the equation
\[
(c_1, c_2) = (c_1, c_2)N^t + (e_1, e_2) + \frac{1}{r}(b_1a_2-b_2a_1)(p, q),
\]
where $e_i =\frac{1}{2} n_{i1}(n_{i1}-1)a_1b_1+\frac{1}{2}n_{i2}(n_{i2}-1)a_2b_2+n_{i1}n_{i2}b_1a_2$ and $N^t$ denotes the transpose of $N$.

In this case $\Gamma$ is generated by the transformations
\begin{align*}
\gamma_0(x, y) &= (\alpha x, y + t),\\
\gamma_i(x, y) &= (x + a_i, y + b_ix + c_i),\,\, (i = 1, 2),\\
\gamma_3(x, y) &= (x, y + r^{-1}(b_1a_2-b_2a_1)).
\end{align*}
This group is discrete and acts properly and discontinuously on $\Hip \times\C$ 
and the quotient we obtain is the compact complex surface $S^+_{N,p,q,r,t}$.

Also in this case, there is only one Riccati foliation on $\Proj(TS^+_{N,p,q,r,t})$ which has
the form $dz$. Moreover, the monodromy group is generated by
$f_0(z) =\frac{z}{\alpha}$ and $f_i(z) =\frac{z}{1+b_iz}$, $i=1,2$. 

Thus we conclude that there are no regular pencils $\Po$ nor regular $k$-webs $\W$ on 
a Inoue surface, for $k\geq 3$.

 \subsubsection{ Elliptic surfaces over a Riemann surface of genus $g\geq 2$, with 
 odd first Betti number.}

The existence of a holomorphic affine structure on a elliptic surface over a Riemann 
surface of genus $g\geq 2$, of odd first Betti number, is a result due to Maehara 
\cite{Maehara}. The global geometry of affine holomorphic structures and holomorphic affine 
connections wi\-thout torsion on these surfaces are studied in \cite{Klingler} and \cite{Dumitrescu} 
respectively.

Up to a finite covering and a finite quotient, this surface $M$ is constructed as follows. 
Let $\Gamma$ be a discrete torsion-free subgroup of $\psl (2, \R)$ such that 
$\Sigma=\Gamma \setminus \Hip$, where
$\Hip$ denotes the Poincar\'e half-plane. If we think of $\Gamma$ as a subgroup of $\SL (2, \R)$,  
then the action of an element 
$\gamma=\left(\begin{array}{cc}
a &b\\
c &d 
\end{array}\right)
\in \SL (2, \R)$ of $\Gamma$ on $\C\times\Hip$ is given by the following formula: 
$\gamma(x,y) = (x + log (cy + d), \gamma y)$, for all $(x,y)\in\C\times\Hip$, where 
 $\log$ denotes a branch of the logarithm function and the action of $\gamma$ on $\Hip$ comes from the canonical action of $\SL(2,\R)$ on $\Hip$. See \cite{Klingler} for more details.

The quotient of $\C\times\Hip$ by the action of $\Gamma$ is the complex 
surface compact $M$.

A torsion-free flat connection on $M$ has the form
\[
\theta= \left(\begin{array}{cc}
dx&0\\
dy&dx
\end{array} \right)+\left(\begin{array}{cc}
0&f(y)dy\\
0&h(y)dy
\end{array} \right),
\]
where $h(\gamma y)=h(y)(cy+d)^2$
and $f(\gamma y)=f(y)(cy+d)^4+ch(y)(cy+d)^3$, 
for all $\gamma=\left(\begin{array}{cc}
a &b\\
c &d 
\end{array}\right)
\in \Gamma$. Therefore the Riccati foliations on $\Proj(TM)$ 
are of the form 
\begin{equation}\label{eq:Riccatielliptic}
\omega=dz+dy+h(y)dyz-f(y)dyz^2.
\end{equation} 

In this case we did not succeed in calculating the monodromy but we prove the following lemma.

\begin{lem} There are no regular pencils $\Po$ on $M$. In particular, there are 
no regular $k$-webs $\W$ on $M$, for $k\geq 3$.
\end{lem}
\begin{proof}
Suppose that there exists a regular pencil $\Po=\{\Fo_{t}\}_{t\in\Proj^1}$ on $M$. 
By \cite[Lemma 4.3]{Puchuri} we may assume that $\Fo_{\infty}$ is the foliation tangent 
to the elliptic fibration on $M$. Let $\pi:\C\times\Hip\to M$ be the universal covering and  
$\tilde{\Po}=\pi^*\Po=\{\tilde{\Fo}_{t}\}_{t\in\Proj^1}$ be the pencil on $\C\times\Hip$. 
Since $\tilde\Fo_0$ and $\tilde\Fo_{\infty}$ are transversal we have 
$\tilde\Fo_0=\{dx+B(x,y)dy=0\}$ and $\tilde\Fo_{\infty}=\{dy=0\}$, $B\in\Ol(\C\times \Hip)$. 
So, $\tilde{\Po}$ is induced by $\omega_t=dx+B(x,y)dy+tA(x,y)dy$, for some $A\in\Ol^*(\C\times \Hip)$. 
Thus the Riccati foliation is induced by 
$\omega=dz+\frac{dA}{A}z+(B\frac{dA}{A}-dB)z^2$, which is in contradiction with 
the equation (\ref{eq:Riccatielliptic}).   
\end{proof}

\section{Applications}\label{section4} 

The following theorem was proved in \cite[Theorem 4.6]{Puchuri} by using foliation techniques. 
\begin{teo} \label{pencilflat}
If $M$ is compact and $\Po$ is a re\-gu\-lar pencil on $M$, then $M$ is either 
a complex torus $M=\C^2/\Gamma$ or a Hopf surface $M=H_{a,l}$.  

Moreover, $\Po$ is the only regular pencil on $M$, and it is generated by
$\{dx+tdy\}_{t\in\Proj^1}$ on $\C^2$ (in the first case) or on 
$\C^2\setminus \{(0,0)\}$ (in the second case). 
\end{teo}

\begin{proof}
Let $\Ho$ be the Riccati foliation on $\Proj(TM)$ induced by $\Po$. By Proposition 
\ref{pencilRiccati} we have that $Mon(\Ho)$ is trivial. Using the calculations of the  
monodromy from Subsection \ref{calculations}, see Lemmas \ref{L:torustrivial} and 
\ref{L:hopffinite}, we conclude the proof. 
\end{proof}

\begin{teo} \label{T:Riccatifinite} 
If $M$ is compact and $\Ho$ is a Riccati foliation on $\Proj(TM)$ with finite monodromy,     
then $M$ is either a complex torus $\C^2/\Gamma$, a hyperelliptic surface, or 
a Hopf surface. 

Moreover, in the first case $\mon(\Ho)$ is trivial, while in the last 
two cases it is cyclic.   
\end{teo}

\begin{proof}
This follows readily from the procedure we employed when computing monodromy in the 
previous subsection, see Lemma \ref{L:torustrivial}, Lemma \ref{L:hyperfinite} and Lemma \ref{L:hopffinite}.    

\end{proof} 

\begin{teo}
If $M$ is compact and $\W$ is a regular $k$-web on $M$, $k\geq 3$,   
then $M$ is either a complex torus, a hyperelliptic surface, or, up to a finite unramified cover, 
the Hopf surface $H_{a,l}$. 

Moreover, $\W\subset \Po$ up to a finite unramified cover, where $\Po$ is 
the only pencil on $M$.    
\end{teo}

\begin{proof}
By Proposition \ref{kwebs}, there is a Riccati foliation $\Ho$ on $\Proj(TM)$ induced by $\W$
with finite monodromy. So we can apply Theorem \ref{T:Riccatifinite}.
Now up to a finite unramified cover we can suppose that the monodromy is trivial and 
$\W = \Fo_1\boxtimes  \cdots \boxtimes \Fo_{k}$ is a 
(completely decomposable) regular $k$-web on $M$. 
To complete the proof we apply Theorem \ref{pencilflat} and proceed as we did when 
computing monodromy in Subsection \ref{calculations}.
\end{proof}

Finally, we point out that the classification of regular $2$-webs follows directly from
\cite[Theorem C]{Beauville}.


\section*{Acknowledgements}
The author wishes to express his deepest gratitude to Frank Loray for lots of fruitful discussions 
about the content of this paper, and also for reading the drafts and suggesting improvements.
 The author also wishes to thank Cesar Hilario for the suggestions and comments on the manuscript.
 The author acknowledge financial support from CAPES/COFECUB 
 (Ma 932/19 "Feuilletages holomorphes et int\'eration avec la g\'eom\'etrie" / process number 88887.356980/2019-00). 
The author is grateful to the Institut de Recherche
en Math\'ematique de Rennes, IRMAR and the Universit\'e de Rennes 1 
for  their hospitality and support. The author is supported by FAPERJ (Grant number E-26/010.001143/2019).

\section*{Declarations}

{\bf Data availability} Data sharing not applicable to this article as no datasets were generated
or analysed during the current study.

\end{document}